\documentclass[10pt]{article}
\usepackage{amsmath}
\usepackage{graphicx}
\usepackage{authblk}
\usepackage{amsfonts}

\usepackage{amsthm}
\usepackage{amsfonts}
\usepackage{algorithm}
\usepackage{algorithmic}
\usepackage{color}

\newtheorem{theorem}{Theorem}

\newtheorem{definition}{Definition} 

\newtheorem{remark}{Remark}

\begin{document}

\title{Counting Roots of a Polynomial in a Convex Compact Region by Means of Winding Number Calculation via Sampling}

\author{Vitaly Zaderman$^{[1]}$ and
Liang Zhao$^{[2]}$
\\
\and
$^{[1]}$
Ph.D. Program in Mathematics,
 
The Graduate Center of the City University of New York, 

New York, NY 10016, USA  

vza52@aol.com 

$^{[2]}$
 Department of Computer Science,
 
The Graduate Center of the City University of New York, 

 New York, NY 10016, USA 
 
liang.zhao1@lehman.cuny.edu
}

\date{}
\maketitle

\begin{abstract}
In this paper, we propose a novel efficient algorithm for calculating winding numbers, aiming at counting the number of roots of a given polynomial in a convex region on the complex plane. This algorithm can be used for counting and exclusion tests in a subdivision algorithms for polynomial root-finding, and would be especially useful in application scenarios where high-precision polynomial coefficients are hard to obtain but we succeed with counting already by  using polynomial evaluation with lower precision. We provide the pseudo code of the algorithm, proof of its correctness  as well as estimation of its complexity.
\end{abstract}


\section{Introduction}
Let
\begin{equation}\label{eqn:poly}
    p(z) = \sum_{k=0}^d p_k z^k = p_d\prod_{j=1}^d(z-z_j), p_d\neq 0
\end{equation}
be a polynomial of degree $d$ with
real or complex
coefficients. Counting its roots (with their multiplicity) in a fixed domain
(such as an interior of a polygon or a disc)
is a fundamental problem with an important 
application to devising efficient root-finders for $p(z)$
on the complex plane, particularly subdivision algorithms,
proposed by Hermann Weyl in \cite {W24} and then extended and improved 
in 
\cite{HG69}, \cite{H74}, \cite{R87}, \cite{P00}, \cite{BSSXY16}, 
and \cite{BSSY18}
\footnote {The authors of \cite{HG69}, \cite{H74}, \cite{P00} called it Quadtree algorithm, and under that name it was extensively used in computational geometry.}
and recently implemented in \cite {IPY18}.

 We propose a new algorithm for counting the roots in a fixed
 convex
  region on the complex plane
 by expressing their number as the winding number computed along the boundary of the region,
 provided that the boundary was sufficiently isolated from the roots of $p(z)$.
 
 Winding number algorithms have been proposed for counting roots in a disc
 as parts of root-finding algorithms 
 by Henrici and Gargantini in \cite{HG69}, then by Henrici in  \cite{H74}
 and by Renegar in  \cite{R87}. Pan in \cite{P00} used root-radii algorithm by 
 Sch{\"o}nhage   \cite{S82} for counting roots in a disc, and 
 Becker et al. in
  \cite{BSSXY16} 
and \cite{BSSY18} performed counting based on Pellet's theorem.

Our  winding number computation shares some techniques with
 the algorithms of  \cite{H74}  and  \cite{R87}, 
but there the algorithms have only devised in the special case of a disc rather than an arbitrary convex compact region, and unlike these papers
we ensure numerical stability of our 
computation of the winding number. 
Another method using insertion technique, but not requiring isolation of the input region 
has been
proposed by Zapata and Martin in \cite{ZM12}, \cite{ZM14}. 
We evaluate an input polynomial $p(z)$
at some additional auxiliary points 
that we insert  a priori on the boundary of the input
 region. In this way we made our parametrization is smooth
  on the associated sub-segments of the boundary curve.
  
Our proposed root-counting algorithm has the following computational advantages:
\begin{itemize}
    \item It does not involve polynomial coefficients: only polynomial evaluations are required. This is especially useful when polynomial evaluations can be provided as a fast "black box".
    \item Computational precision can be kept low: the algorithm 
     outputs the winding number correctly as long as polynomial evaluations are precise enough to indicate correctly     
     the quadrant of the complex plane in which the
     values of the polynomial lies.
    \item Besides evaluating polynomials, only integer calculations are involved.
\end{itemize}

We present our algorithm in the next section and then continue the paper in section \ref{sec:3}
by proving its correctness.

\section{Winding Number Calculation via Sampling}
Suppose that $p(z)$ is a polynomial of
 Eqn. \eqref{eqn:poly}, $\gamma:[0, 1]\rightarrow \mathbb{C}$ is a simple convex closed piece-wise  smooth curve, and $\Gamma$ 
 is the region enclosed by $\gamma$. The winding number 
 $\omega_{p\circ\gamma}$ of a curve $p\circ \gamma$ is the 
 number of counterclockwise turns that $p(\gamma(t))$ makes around the origin as $t$ increases from $0$ to $1$. Namely, 
\begin{equation}\label{eqn:wn}
    \omega_{p\circ\gamma} = \frac{1}{2\pi}\oint_{\gamma}\frac{p'(z)}{p(z)}dz = \frac{1}{2\pi}\int_0^1 \frac{p'(\gamma(t))\gamma'(t)}{p(\gamma(t))}dt.
\end{equation}
Hereafter we write 
$\omega:=\omega_{p\circ\gamma}$
 omitting the subscript $p\circ\gamma$.
  It is  well-known by principle of argument, 
  that if $(p\circ\gamma)(t)\neq 0$ for all $t\in[0, 1]$, then the winding number $\omega$ is a well-defined integer
   equals to the number of the 
   roots of $p(z)$ inside the 
   region bounded by $\Gamma$.

In this paper we aim 
at developing algorithm that calculates winding numbers of $p\circ\gamma$, where $p$ is a univariate polynomial whose roots lie 
reasonably far 
from $\gamma$. In particular, such algorithm can be applied to circles, squares or polygons. Before diving into the details of the algorithm, we should clarify the assumptions  
about the curve $\gamma:[0, 1]\rightarrow\mathbb{C}$.

{\bf Assumptions on $\gamma$}. 
\begin{enumerate}
    \item $\gamma$ is
    the boundary of a connected convex region $\Gamma$ on the complex plane. 
  
    It
    is a convex closed curve, i.e., $\gamma(0) = \gamma(1)$.
    \item There exists the continuous derivative $\gamma'(t)$ except for
    $t$ 
    lying in a finite subset $T\subset [0, 1]$ (which is 
    relatively small).
     \item Furthermore the derivative $\gamma'(t)$ is bounded by $L$ from above, that is,
    \begin{equation}
        L = \max_{t\in[0,1]\backslash T}|\gamma'(t)|
    \end{equation} 
    \item $\gamma$ is $\frac{2}{3}r$-isolating the roots of polynomial
    
     $p(z)$, meaning that 
    the minimum distance between a point on the curve $\gamma$ and a root of $p(z)$ is 
at least    
    $\frac{2}{3}r$ where $r$ denotes the minimal distance between the
    
     origin and the curve $p\circ \gamma$, that is,
    \begin{equation}
        r = \min_{t\in[0,1]}|(p\circ\gamma)(t)|.
    \end{equation} 
    In particular $(p\circ\gamma)(t)\neq 0$ for all $t\in[0, 1]$.
\end{enumerate}
    
\begin{remark}
  {
For a convex domain with a center 
(which covers a disc and an interior of a rectangle as particular cases)
we can define its dilation with a coefficient $\theta>1$.
If the number of roots of $p(z)$ in the domain is invariant in its dilation
with coefficients $\theta$ and  $1/\theta$, then we call the domain 
$\theta$-isolated. We can square isolation coefficient $\theta$ by  performing 
     Dandelin's
 root-squaring iteration}
     $p(z)\rightarrow (-1)^d~
 p(\sqrt z~)p(-\sqrt {z}~)$ 
 (cf. \cite{Ho59} ).
 $s$ iterations
   $$p_0(z)=p(z),~p_{j+1}(z)=(-1)^d~
 p_j(\sqrt z~)p_j(-\sqrt {z}~),~j=0,1,\dots,s$$ 
 change that coefficient into $\theta^{2^s}$ .
 \end{remark}
    
The core idea of our winding number algorithm is to compute the number of turns of $\Gamma$ 
around $0$. We do it by computing polynomial in finite number 
 of points
   $t_0,...,t_N\in[0, 1]$. More precisely we  
   correctly compute the number of roots in a given region if
   for every $i$ 
   the actual value $p(\gamma(t_i))$ and the computed value of $p(x)$ at the point $\gamma(t_i)$ lie in the same quadrant  on the complex plane, labeled by the following
  integers $m(p(\gamma(t_i))$.
\begin{definition}

 {
Given polynomial $p(z)$, closed curve $\gamma:[0,1]\rightarrow\mathbb{C}$, and $t\in[0,1]$, the 
quadrant label
}
 $m_t = m((p\circ\gamma)(t))$ 
 { 
 is defined as
 }
\begin{equation}
    m_t = m((p\circ\gamma)(t)) =
    \left\{
    \begin{array}{cc}
        0 & \textit{   if } Re((p\circ\gamma) (t))> 0, Im((p\circ\gamma)(t))\ge 0 \\
        1 & \textit{   if } Re((p\circ\gamma)(t))\le 0, Im((p\circ\gamma) (t))> 0 \\
        2 & \textit{   if } Re((p\circ\gamma)(t))<0, Im((p\circ\gamma)(t)  \le 0 \\
        3 & \textit{   if } Re((p\circ\gamma)(t))\ge0, Im((p\circ\gamma)(t))<0 \\
    \end{array}
    \right.
\end{equation}
\end{definition}

We simplify the notation  
 letting the integers $m_0, ..., m_N$  denote the quadrant labels
 for a sequence $0 = t _0<t_1<\cdots <t_N\le 1$. 

We are going to prove 
that the winding number increases by 1 (respectively, decreases by 1) whenever a sub-sequence $(m_0,...,m_l)$ goes through all four quadrants counterclockwise (respectively, clockwise). 

$1\rightarrow 2\rightarrow 3\rightarrow 3\rightarrow 0\rightarrow 1$
is an example of a full counterclockwise cycle, and  
$3\rightarrow 2\rightarrow 1\rightarrow 2\rightarrow 1\rightarrow 0\rightarrow 3$
is an example of a full clockwise cycle. 
 
Notice 
that in the latter example the labels go counterclockwise at some point (the $1\rightarrow 2$ part), but  do not complete a full counterclockwise cycle and thus make no impact on the value of winding number. 

To calculate the number of cycles in quadrant labels, we take the difference of each quadrant label with its preceding label modulo 4. For example, the difference between
label 2 and its preceding label 1
is $2-1 \equiv 1(\textit {mod }4)$; the difference between label 0 and its preceding label 3 is also 1, since $0-3=-3\equiv 1(\textit{mod }4)$. 

 Notice that for a counterclockwise cycle, the overall sum of these differences must equal 4 (as there must be 4 net increases in quadrant labels); for a clockwise cycle, the overall sum of 
the label differences must be -4 (as there must be 4 net decreases in quadrant labels). As a result, if we construct sequence $m(0), ..., m(N)$ where $m(0) = m_0$ and $m(k)$ for $k=1,...,N$ are chosen such that $m(k) - m(k-1)\in \{0, 1, 2, 3\}$ and $m(k) - m(k-1)\equiv m_k - m_{k-1} (\textit{mod } 4)$, then $(m(N) - m(0))/4$ will be the number of counterclockwise cycles minus the number of clockwise cycles.

In order to establish the link between winding number and the cycles of quadrant labels, we need to eliminate two possibilities: 1) a full cycle of the curve that does not correspond to a full cycle of quadrant labels (this may happen if the sampled points are too far apart, for instance only three first-quadrant points from a cycle are sampled, showing labels $0\rightarrow 0\rightarrow 0$), and 2) we  cannot  
determine whether a full cycle of quadrant labels is a clockwise or counterclockwise 
cycle
(this may happen when two consecutive quadrant labels differ by more than 1, e.g., if $0\rightarrow 2\rightarrow 0$).  Our winding number algorithm ensures that the points are sampled properly so that neither bad scenario will occur, and so the winding number can be calculated correctly as
\begin{equation}
    \omega = \frac{m(N) - m(0)}{4}.
\end{equation}

\begin{algorithm}\label{alg:WN}
\caption{The Winding Number Algorithm}
\begin{algorithmic}[1]
\REQUIRE A polynomial $p(z)=\sum_{k=0}^d p_kx^k$, a region $\Gamma$ with boundary parametrized as a piece-wise smooth curve $\gamma:[0,1]\rightarrow \mathbb{C}$, $r>0$, $L>0$.
\ENSURE A positive integer $\omega$ such that if $\gamma$, $r$, $L$ satisfy Assumption 1-4, then $\omega$ equals to the winding number of $p\circ\gamma$.
\STATE Sample $N=\lceil{\frac{12dL}{\pi r}}\rceil + |T|$ points $0 = 
t_0<t_1<\cdots<t_N \leq 1$ such that $T\subset\{t_i:0\le i\le N\}$ and $t_i - t_{i-1}\le \frac{\pi r}{12dL}$ for all $i=1,...,N+1, t_{N+1}:=t_0$.
\STATE $m_0\leftarrow $ the quadrant label of $(p\circ\gamma)(t_0)$.
\FOR{i=1 to N}
\STATE $m_i\leftarrow $ the quadrant label of $(p\circ\gamma)(t_i)$
\STATE Choose $m(i)$ such that $\{0,1,2,3\}\ni m(i) - m(i-1)\equiv m_i - m_{i-1}($mod $4)$.
\ENDFOR  \ i
\RETURN $\frac{m(N) - m(0)}{4}$.
\end{algorithmic}

\end{algorithm}

\section{Correctness of 
the Winding Number \\
Algorithm}
\label{sec:3}

In this section we 
prove that 
our algorithm indeed produces 
correct winding number.

\begin{theorem}

 {
For  a degree $d$ univariate polynomial $p(z)$, a parametrized curve
 satisfying Assumption 1 - 4,
and a sequence $0=t_0<t_1<\cdots <t_N\le 1$ such that $|t_i - t_{i-1}| \le \frac{\pi r}{12dL}$ for all $i=1,...,N+1, t_{N+1}:= t_0$, construct using Algorithm 1 
a sequence of integers $m(0),...,m(N)$ such that $m(0)=m_0$, $m(i)-m(i-1)\in\{0,1,2,3\}$, and $m(i)-m(i-1)\equiv m_i-m_{i-1}(\textit{mod }4)$ for $i=1,...,N$, where $m_i$ is the quadrant label of $(p\circ\gamma)(t_i)$. Then the winding number $\omega$ of $p(z)$ along curve $\gamma$ 
is equal to
\begin{equation}
    \omega = \frac{m(N) - m(0)}{4}.
\end{equation}
}
\end{theorem}

\begin{proof} 

 {
On each segment $[t_i, t_{i-1}], \gamma (t)$ is smooth. If a sequence of consecutive labels $m(i), m(i+1), ..., m(j)$ completes a counterclockwise cycle, then the sum of differences must equal to 4, i.e., 
\begin{equation}
   m(j) - m(i) = \sum_{k=i}^{j-1} (m(k+1) - m(k)) = 4.
\end{equation}
Similarly, a sequence of labels representing a clockwise cycle must satisfy $m(j) - m(i) = -4$. Thus the overall sum $\frac{m(N) - m(0)}{4}$ 
is equal to the number of counterclockwise cycles minus the number of clockwise cycles. Given this property, it suffices to show that for any $i=1,...,N$ 
it holds that
\begin{enumerate}
    \item It is impossible that the curve $p\circ\gamma$ can complete a full turn in $[t_i, t_{i+1}]$, that is, 
    \begin{equation}\label{eqn:cond1}
        \frac{1}{2\pi}\int_{t_{i-1}}^{t_{i}} \frac{p'(\gamma(t))\gamma'(t)}{p(\gamma(t))}dt < 1
    \end{equation}.
    \item The quadrant labels $m_i$ 
    differs from $m_{i-1}$ by at most $1$, that is,
    \begin{equation}
        |m_i - m_{i-1}|\le 1.
    \end{equation}
\end{enumerate}
}

Proof of claim 
1. Recall that $p(z) = p_d  \prod_{j=1}^d(z-z_j)$ and that
\begin{equation}
    \frac{p'(z)}{p(z)}=\sum_{j=1}^d\frac{1}{z-z_j}.
\end{equation}
We will show that the integral in Eqn. \eqref{eqn:cond1} is less than $2\pi$.  
It follows
that
\begin{equation}
    \aligned
    \big|\int_{t_{i-1}}^{t_{i}} \frac{p'(\gamma(t))\gamma'(t)}{p(\gamma(t))}dt\big|
    \le&\int_{t_{i-1}}^{t_{i}} \big|\frac{p'(\gamma(t))}{p(\gamma(t))}\big||{\gamma'(t)}|dt\\
    \le&L\int_{t_{i-1}}^{t_{i}}\sum_{j=1}^d\big|\frac{1}{\gamma(t) - z_j}\big|dt\\
    \le&L\int_{t_{i-1}}^{t_{i}}\frac{3d}{r}dt\\
    =& \frac{3dL}{r}(t_i - t_{i-1})\\
    \le& \frac{3dL}{r}\cdot\frac{\pi r}{12dL}\\
    =& \frac {\pi}{4} < 2\pi.
    \endaligned
\end{equation}
This verifies Eqn. \eqref{eqn:cond1}.

Proof of claim
2. If $m_i$ differs from $m_{i-1}$ by more than one, then the path $(p\circ\gamma)(t)$ would cross both the real axis and the imaginary axis as $t$ increases from $t_{i-1}$ to $t_i$. As a consequence, the argument of $(p\circ\gamma)(t)$ would change at least 
by $\pi/4$. Since 
\begin{equation}
    arg((p\circ\gamma)(t)) = \sum_{j=1}^d arg((\gamma(t)-z_j),
\end{equation}
 {
there exists at least one $j$ such that $arg(\gamma(t_i)-z_j)$ differs from $arg(\gamma(t_{i-1})-z_j)$ by more than $\pi/(4d)$. Next we will show that this is impossible,  
because according to the choice of samples, $\gamma(t_i)$ is very close to $\gamma(t_{i-1})$.
}
 {
On one hand,
}
\begin{equation}
    \aligned
    |\gamma(t_i) - \gamma(t_{i-1})| \le L|t_i - t_{i-1}| \le \frac{\pi r}{12d}.
    \endaligned
\end{equation}

 {
On the other hand,
 both
 }
 {  $|\gamma(t_i) - z_j|$ and $|\gamma(t_{i-1}) - z_j|$ are at least $2r/3$ 
  and their arguments differ 
by at least
}
 $\pi/4d$. Let $\theta_1 = arg(\gamma(t_i) - z_j)$ and $\theta_2=arg(\gamma(t_{i-1}) - z_j)$, \ $\theta_1 \neq \theta_2$
 {
then
}
\begin{equation}
    \aligned
    |\gamma(t_i) - \gamma(t_{i-1})| = & 
    |(\gamma(t_i) - z_j) - (\gamma(t_{i-1}) - z_j)|\\
    \ge& |\frac{2r}{3}e^{\theta_1 i} - \frac{r}{3}e^{\theta_2 i}|\\ 
      =& \frac{2r}{3}|e^{(\theta_1 - \theta_2)i} - 1|\\
    >
    & \frac{2r}{3}\cdot|\theta_1-\theta_2|/2 \\
    \ge& \frac{\pi r}{12d}.
    \endaligned
\end{equation} \\
{
A contradiction proves the claim.
}

\end{proof}

{\bf Computation Complexity} The complexity of the algorithm is dominated by the evaluations of the polynomial at $N$ sampled points. Besides polynomial evaluation, the algorithm only requires arithmetic of small integers (mostly less than 8). The value of $N$ is proportional to the Lipschitz bound $L$ defined in Assumption 3. Thus the speed of the algorithm is determined by how fast it can obtain polynomial evaluations at sampled points. 
If the region is the unit disc $\{z:~|z|\le 1\}$, then we can evaluate $p(z)$  
at $2^h$ equally-spaced points on the unit boundary circle
 $\{z:~|z|=1\}$ and by using
 FFT, would correctly compute the number of roots of $p(z)$ in the disc at a arithmetic cost in $\tilde{O}(dL)$, which means
  ${O}(dL)$ up to
poly-logarithmic factors in $dL$.


\bigskip


\noindent {\bf Acknowledgements:}
Our research has been supported by the NSF Grant CCF--1563942 , NSF Grant CCF-1733834, and the PSC CUNY Award 69813 00 48.
  


\begin{thebibliography} {\hspace{1cm}}



\bibitem{BSSXY16}
Becker, R., Sagraloff, M., Sharma, V., Xu, J., Yap, C.: Complexity analysis of root
clustering for a complex polynomial. In: Proceedings of the ACM on International
Symposium on Symbolic and Algebraic Computation. pp. 71-78. ACM (2016)

\bibitem{BSSY18}
Becker, R., Sagraloff, M., Sharma, V., Yap, C.: A near-optimal subdivision algorithm
for complex root isolation based on the Pellet test and Newton iteration. J.
Symb. Comput. 86, 51-96 (2018)

\bibitem{H74}
Henrici, P.: Applied and computational complex analysis. vol. 1, power series,
integration, conformal mapping, location of zeros. John Wiley (1974)

\bibitem{HG69}
Henrici, P., Gargantini, I.: Uniformly convergent algorithms for the simultaneous
approximation of all zeros of a polynomial. In: Constructive Aspects of the Fundamental
Theorem of Algebra. pp. 77-113. Wiley-Interscience New York (1969)

\bibitem{Ho59}
Householder, A.S.: Dandelin, Lobachevski, or Graeffe? Am. Math. Mon. 66(6), 464-466 (1959)

\bibitem{IPY18}
Imbach, R., Pan, V.Y., Yap, C.: Implementation of a near-optimal complex root
clustering algorithm. In: International Congress on Mathematical Software. pp.
235-244. Springer (2018)

\bibitem{P00}
Pan, V.Y.: Approximating complex polynomial zeros: modified Weyl's quadtree
construction and improved newton's iteration. J. Complex. 16(1), 213-264 (2000)

\bibitem{R87}
Renegar, J.: On the worst-case arithmetic complexity of approximating zeros of
polynomials. J. Complex. 3(2), 90-113 (1987)

\bibitem{S82}
Sch{\"o}nhage, A.: The fundamental theorem of algebra in terms of computational complexity. Manuscript. Univ. of Tubingen, Germany (1982)
\bibitem{W24}
Weyl, H.: Randbemerkungen zu hauptproblem der mathematik. Mathematische
Zeitschrift 20, 131-150 (1924)

\bibitem{ZM12}
Zapata, J.L.G., Martin, J.C.D.: A geometric algorithm for winding number computation
with complexity analysis. J. Complex. 28(3), 320-345 (2012)

\bibitem{ZM14}
Zapata, J.L.G., Martin, J.C.D.: Finding the number of roots of a polynomial in
a plane region using the winding number. Comput. Math. Appl. 67(3), 555-568
(2014)





\end {thebibliography}
\end{document}